\documentclass[11pt]{article}
\usepackage[utf8]{inputenc}
\usepackage[T1]{fontenc}
\DeclareUnicodeCharacter{0229}{\k{e}}
\usepackage{lmodern}
\usepackage{subfiles}
\usepackage{enumitem}
\setenumerate{topsep=6pt,ref={\normalfont(\roman*)},label={\normalfont(\roman*)}, itemsep=0pt} 

\usepackage{amsfonts}
\usepackage{amsthm}
\usepackage{amsmath}
\usepackage{amssymb}
\usepackage{amscd}
\usepackage{mathrsfs}
\usepackage{mathtools}
\usepackage{bbm}
\usepackage{esint}

\usepackage[margin=2.7cm]{geometry}
\usepackage{setspace}
\usepackage{indentfirst}
\usepackage{graphicx}
\usepackage{graphics}
\usepackage{lscape}
\usepackage{pgf,tikz}
\usepackage{tikz-cd}
\usepackage{color}
\usepackage{pict2e}
\usepackage{epic}
\usepackage{epstopdf}
\usepackage{titlesec, titlefoot}
\titleformat{\section}[block]{\Large\bfseries\filcenter}{\thesection}{1em}{}
\titleformat{\part}[block]{\LARGE\bfseries\filcenter}{Part \thepart.}{0.5em}{}
\usepackage{commath}
\usepackage{float}
\usepackage{caption}
\usepackage{etoolbox}
\usepackage{combelow}

\usepackage[nocompress]{cite}
\usepackage[hidelinks,bookmarksdepth=3]{hyperref}
\hypersetup{bookmarksopen=true} 

\graphicspath{{./Pictures/}}
\allowdisplaybreaks

\expandafter\def\expandafter\normalsize\expandafter{%
\normalsize
\setlength\abovedisplayskip{6pt}
\setlength\belowdisplayskip{6pt}
\setlength\abovedisplayshortskip{6pt}
\setlength\belowdisplayshortskip{6pt}
}

\theoremstyle{plain}

\renewcommand*\thesection{\arabic{section}}

\newtheorem{theorem}{Theorem}
\newtheorem{lemma}[theorem]{Lemma}
\newtheorem*{lemma*}{Lemma}

\newtheorem{corollary}[theorem]{Corollary}

\theoremstyle{definition}

\expandafter\let\expandafter\oldproof\csname\string\proof\endcsname
\let\oldendproof\endproof
\renewenvironment{proof}[1][\proofname]{%
\oldproof[\upshape \bfseries #1]%
}{\oldendproof}

\makeatletter
\def\@makechapterhead#1{%
\vspace*{50\p@}%
{\parindent \z@ \raggedright \normalfont
\interlinepenalty\@M
\Huge\bfseries  \thechapter.\quad #1\par\nobreak
\vskip 40\p@
}}
\makeatother

\let\textcaron\v
\renewcommand{\v}[1]{\ifmmode\check{#1}\else\textcaron{#1}\fi}

\def \R {\mathbb{R}}
\def \Z {\mathbb{Z}}
\def \C{\mathbb{C}}

\def \D{\textup{D}}

\def \d{\,\textup{d}}

\def \p{\partial}
\def \mc{\mathcal}
\def \mb{\mathbb}

\def \tp{\textup}

\renewenvironment{thebibliography}[1]{
  \begin{oldthebibliography}{#1}
    \setlength{\itemsep}{0.5pt}
    \setlength{\parskip}{0.5pt}
}{
  \end{oldthebibliography}
}

\begin{document}

	\title{\textbf{Sharp Higher Integrability Theory, Part I:\\ The $L^p$-norm of the Beurling--Ahlfors transform}}
		
	\author{
  \parbox{0.95\textwidth}{\centering
    \mbox{Andrea Agazzi},
    \mbox{Kari Astala},
    \mbox{Rodrigo Ba\~nuelos},
    \mbox{Giuseppe Bruno},
    \mbox{Gabriele Cassese},
    \mbox{Daniel Faraco},
    \mbox{Andr\'e Guerra},
    \mbox{Mahdi Hormozi},
    \mbox{Bernd Kirchheim},
    \mbox{Aleksis Koski},
    \mbox{Vjekoslav Kova\v{c}},
    \mbox{Jan Kristensen},
    \mbox{Federico Pasqualotto},
    \mbox{Istv\'an Prause},
    \mbox{Eero Saksman},
    \mbox{Vladim\'ir \v{S}ver\'ak},
    \mbox{L\'aszl\'o Sz\'ekelyhidi},
    \mbox{Riccardo Tione},
    \mbox{Alexander Volberg},
    \mbox{Hang Yang}
  }
}
				
	\date{}
	
	\maketitle
	
	\begin{abstract}
    We confirm a conjecture posed by T.\ Iwaniec in 1982: the $L^p$ norm of the Beurling--Ahlfors transform is $p-1$ if $p\geq 2$, and $(p-1)^{-1}$ if $1<p\leq 2$.
    \end{abstract}

\begingroup
\renewcommand{\abstractname}{AI disclosure}
\begin{abstract}
The key ideas of this paper were generated by ChatGPT 6.0 Astra, without substantial human assistance. The proof presented here does not belong to a single author, but rather to a community, imperfectly represented in the list of authors of this manuscript, whose efforts across the last four decades make the solution of this beautiful problem finally possible.
\end{abstract}
\endgroup

\section{Introduction}

This work is the first in a series of papers, and it serves two main purposes: one of them is mathematical, and the other is meta-mathematical. The meta-mathematical purpose is our attempt at demonstrating how collaboration and cooperation can enable communities of mathematicians to develop collective ownership and understanding of deep, important mathematical results within their respective subfields. 


The mathematical purpose is to develop a comprehensive sharp higher-integrability theory, with applications to Harmonic Analysis, Elliptic PDEs, and Geometric Function Theory. In this first paper in the series we prove Theorem \ref{thm:L-quasiconvexity} below. To state it, we identify $\R^2$ with $\C$ and write a real-linear map $A$ in the form
\[
Az=a_+z+a_-\bar z.
\]
With $|\cdot|$ denoting the operator norm, we then have $|A|=|a_+|+|a_-|$ and $\det A=|a_+|^2-|a_-|^2$. Now define $L\colon \R^{2\times 2}\to \R$ by
\begin{equation}\label{eq:L-definition}
L(A)=
\begin{cases}
|a_+|^2-|a_-|^2&\text{if } |A|\leq1,\\
2|a_+|-1&\text{if } |A|>1.
\end{cases}
\end{equation}
This integrand was discovered independently in \cite{Sverak1990,Burkholder1989}. We prove:

\begin{theorem}\label{thm:L-quasiconvexity}
Let $Q=\R^2/\Z^2$, with $|Q|=1$. For all $A\in\R^{2\times 2}$ and all $\varphi\in C^\infty_{\tp{per}}(Q;\R^2)$ we have
\[
\int_Q\bigl[L(A+\D\varphi)-L(A)\bigr]\d x\geq0.
\]
\end{theorem}

In the language of the Calculus of Variations, Theorem \ref{thm:L-quasiconvexity} asserts that $L$ is \textit{quasiconvex} \cite{Dacorogna2007}. In the case where $\varphi=\D u$ for some $u\in C^\infty_{\tp{per}}(Q)$ this result had been proved in \cite{Guerra2026b}.

It is well-known that Theorem \ref{thm:L-quasiconvexity} has deep consequences, as we briefly explain.
Set $p^*=\max\{p,p/(p-1)\}$ and for $1<p<\infty$ define the Burkholder integrand \cite{Burkholder1984,Burkholder1989}, $B_p\colon\R^{2\times2}\to\R$, by
\begin{equation}\label{eq:Bp-definition}
B_p(A)=
\begin{cases}
\bigl[(p^*-1)|a_+|-|a_-|\bigr](|a_+|+|a_-|)^{p-1}&\text{if } 1<p<2,\\
\bigl[(p^*-1)|a_-|-|a_+|\bigr](|a_+|+|a_-|)^{p-1}&\text{if } 2\leq p<\infty.
\end{cases}
\end{equation}
We also recall that the Beurling--Ahlfors transform $\mc S$ is the Fourier multiplier characterized by
\begin{equation}\label{eq:BA-definition}
\mc S\circ\p_{\bar z}=\p_z,
\end{equation}
see \cite{Astala2009}.
By the arguments in  \cite{Baernstein1997a} (see also \cite{Guerra2026b} for further details), Theorem~\ref{thm:L-quasiconvexity} implies:

\begin{corollary}\label{cor:Bp-BA}
For every $1<p<\infty$, for every $A\in\R^{2\times2}$ and every $\varphi\in C^\infty_{\tp{per}}(Q;\R^2)$,
\[
\int_Q\bigl[B_p(A+\D\varphi)-B_p(A)\bigr]\d x\geq0.
\]
Moreover, the Beurling--Ahlfors transform satisfies
\begin{equation}\label{eq:BA-sharp-norm}
\|\mc S\|_{L^p(\C;\C)\to L^p(\C;\C)}=p^*-1=
\begin{cases}
(p-1)^{-1},&1<p\leq2,\\
p-1,&2\leq p<\infty.
\end{cases}
\end{equation}
\end{corollary}

Previously, quasiconvexity inequalities for the Burkholder function under sign restrictions, together with related lower semicontinuity results, were established in \cite{Astala2022,Astala2024,Astala2012,GuerraKristensen2021}.
Concerning the Beurling--Ahlfors transform $\mc S$,  we note that the lower bound in \eqref{eq:BA-sharp-norm} is classical \cite{Lehto1965} and that the novelty is the upper bound. 
Previous, suboptimal bounds, had been obtained in  \cite{Banuelos2008,Banuelos1995,Borichev2013,Dragicevic2005,Volberg2004}; see also the survey \cite{Banuelos2010} on Burkholder’s foundational inequalities for additional information on this problem.


Corollary \ref{cor:Bp-BA} settles a conjecture due to T.\ Iwaniec \cite{Iwaniec1982}. This conjecture has played a central role in Geometric Function Theory, and Corollary \ref{cor:Bp-BA} implies the sharp higher integrability of derivatives of quasiconformal maps in the plane, originally obtained by complex-analytical methods in \cite{Astala1994}. These methods do not seem to generalize to higher dimensions.

Instead, the proof of Theorem \ref{thm:L-quasiconvexity} is not complex-analytic, and it extends to higher dimensions, implying a similar results for a more general version of the Beurling--Ahlfors transform, defined on forms \cite{Donaldson1989,Iwaniec1993b}, see also \cite{banuelos1997martingale,Duse2020,petermichl2011new}. As anticipated in the remarkable joint works of T.\ Iwaniec with G.\ Martin and C.\ Sbordone, this has further important consequences for the sharp higher integrability theory of quasiconformal mappings in even dimensions \cite{Iwaniec1993b} and for solutions of elliptic PDEs in all dimensions \cite{Iwaniec2001a}, see also \cite{Iwaniec2002}. The proof of Theorem \ref{thm:L-quasiconvexity} can also be generalized in order to encompass the very large family of integrands studied in \cite{Faraco2008,Kirchheim2008,Szekelyhidi2005}, with applications to the study of the Monge--Ampère equation \cite{Ambrosio2011,DePhilippis2014,DePhilippis2023}. These results will be developed and presented in the next articles in this series.

\section{Proof}\label{sec:proof}

For a matrix $A$ with rows $A^1,A^2$, put
\[
J=\begin{bmatrix}0&-1\\1&0\end{bmatrix},\qquad
\det A = \det(A^1,A^2)=JA^1\cdot A^2.
\]
Define
\[
\begin{aligned}
G_z(A^1,A^2)&=\det(z,A^2)+z\cdot(A^1-z),\\
F(A)&=\sup_{|z|=1}G_z(A^1,A^2)=|A^2+JA^1|-1=2|a_+|-1,
\end{aligned}
\]
and note that $F$ is convex. 
\begin{lemma}
We have
\begin{equation}\label{eq:L-row-representation}
L(A)=\sup_{z\in\mc Z(A^1)}G_z(A^1,A^2),\qquad
\mc Z(A^1)=\mb S^1\cup
\begin{cases}
\{A^1\},&\text{if }|A^1|\leq1,\\
\emptyset,&\text{if }|A^1|>1.
\end{cases}
\end{equation}
\end{lemma}

\begin{proof}
Indeed, $G_{A^1}(A^1,A^2)=\det A$ and so
$$
\sup_{z\in\mathcal Z(A^1)}G_z(A^1,A^2)
=
\begin{cases}
\max\{\det A,F(A)\},&|A^1|\leq1,\\
F(A),&|A^1|>1.
\end{cases}
$$
 We note the identities
\[
\det A-F(A)=(1-|a_+|-|a_-|)(1-|a_+|+|a_-|),\qquad
\bigl||a_+|-|a_-|\bigr|\leq|A^1|\leq|a_+|+|a_-|.
\]
If $|a_+|+|a_-|\leq1$, then $|A^1|\leq1$ and $\det A\geq F(A)$. If $|a_+|+|a_-|>1$ and $|A^1|\leq1$, then $|a_+|-|a_-|\leq1$ and $\det A\leq F(A)$. If $|A^1|>1$, then $|a_+|+|a_-|>1$ and $\mc Z(A^1)=\mb S^1$. The conclusion follows.
\end{proof}

\begin{proof}[Proof of Theorem~\ref{thm:L-quasiconvexity}]
Let $u(x)=(u^1,u^2)(x)=Ax+\varphi(x)$, where $\varphi\in C^\infty_\tp{per}(Q)$. 
Note that $L\geq F$ everywhere and so, if $|A|>1$, convexity gives
\[
\int_Q L(\D u)\d x\geq\int_Q F(\D u)\d x\geq F(A)=L(A).
\]
Hence we can assume that $|A|\leq1$, and then $|A^1|\leq1$. Apply Lemma~\ref{lem:scalar-projection} below with $K=\overline B$, the closed unit disk, to $u^1(x)=A^1\cdot x+\eta(x)$: we find $\zeta \in W^{1,\infty}_\tp{per}(Q)$ such that $w(x)=A^1\cdot x+\zeta(x)$ satisfies $|\D w|\leq1$ and either $\D w=\D u^1$ or $|\D w|=1$ almost everywhere. Hence $\D w\in\mc Z(\D u^1)$, and \eqref{eq:L-row-representation} yields
\[
L(\D u)\geq\det(\D w,\D u^2)+\D w\cdot(\D u^1-\D w).
\]
Since the determinant is a null Lagrangian and $w(x)-A^1\cdot x, u^2(x)-A^2\cdot x$ are periodic we have
\[
\int_Q\det(\D w,\D u^2)\d x=\det A.
\]
Inequality \eqref{eq:projection-integral-sign} then gives
\[
\int_Q L(\D u)\d x\geq\det A=L(A),
\]
as wished.
\end{proof}

To complete the proof we recall the following standard lemma, whose proof we include for the sake of completeness.

\begin{lemma}\label{lem:scalar-projection}
Let $K\subset\R^2$ be nonempty, compact and convex and take $p\in K$. For $\eta\in W^{1,\infty}_{\tp{per}}(Q)$, set $v(x)=p\cdot x+\eta(x)$ and let $\zeta$ be the $L^2(Q)$-projection of $\eta$ onto
\[
\mc C=\{\xi\in W^{1,\infty}_{\tp{per}}(Q):p+\D\xi\in K\text{ a.e.}\}.
\]
Then $w(x)=p\cdot x+\zeta(x)$ satisfies either $\D w=\D v$ or $\D w\in\partial K$ almost everywhere, and
\begin{equation}\label{eq:projection-integral-sign}
\int_Q\D w\cdot(\D v-\D w)\d x\geq0.
\end{equation}
\end{lemma}

\begin{proof}
The set $\mc C$ is nonempty, closed and convex in $L^2(Q)$. By \cite[Chapter~I, Lemma~2.1 and Theorem~2.3]{kinderlehrer2000introduction}, the projection $\zeta$ exists uniquely and, with $h(x)=\eta(x)-\zeta(x)$, satisfies
\begin{equation}\label{eq:projection-variational-inequality}
\int_Q h(\xi-\zeta)\d x\leq0\qquad\forall\xi\in\mc C.
\end{equation}

We have $\D w=\D v$ a.e.\ on $\{h=0\}$. To prove the first assertion, it remains to show that $\D w\in\partial K$ almost everywhere on $\{h\ne0\}$, so we can assume that $K$ has non-empty interior. Let
\[
H(y)=\max_{z\in K}z\cdot y
\]
be its support function. Then $H$ is Lipschitz and $\D H\in K$ almost everywhere.
Suppose, towards a contradiction, that $h(x_0)\ne0$, $w$ is differentiable at $x_0$, and $b_0=\D w(x_0)\in\operatorname{int}K$. By changing signs, we may assume $h(x_0)>0$. Choose $\rho>0$ such that $B_\rho(b_0)\subset K$, and then
\[
H(y)\geq b_0\cdot y+\rho|y|\qquad\forall y\in\R^2.
\]
By continuity of $h$ and differentiability of $w$ at $x_0$, we can choose a ball $B_r(x_0)$  where $h>0$  and
\[
|w(x_0+y)-w(x_0)-b_0\cdot y|\leq\frac\rho2|y|\qquad\text{for }|y|\leq r.
\]
For $0<\varepsilon<\rho r/2$, define
\[
c(x)=w(x_0)+\varepsilon-H(x_0-x),
\]
which has gradient in $K$, $c<w$ on $\partial B_r(x_0)$ and  $c(x_0)>w(x_0)$. Thus $(c-w)_+$ has compact support in $B_r(x_0)$ and, extending it by zero outside $B_r(x_0)$, we obtain $\zeta+(c-w)_+\in\mc C$, since $p\cdot x + \zeta+(c-w)_+ = \max \{w,c\}$ in $B_r(x_0)$.
But we have $\int_Q h(c-w)_+\d x>0$, contradicting \eqref{eq:projection-variational-inequality}.

To prove \eqref{eq:projection-integral-sign}, note that $\zeta(\cdot\pm te_i)\in\mc C$ for $i=1,2$ and $t\ne0$. Adding the corresponding inequalities \eqref{eq:projection-variational-inequality} and changing variables gives
\begin{align*}
0&\geq\int_Q h(x)\bigl[\zeta(x+te_i)+\zeta(x-te_i)-2\zeta(x)\bigr]\d x\\
&=-\int_Q\bigl[h(x+te_i)-h(x)\bigr]\bigl[\zeta(x+te_i)-\zeta(x)\bigr]\d x.
\end{align*}
Dividing by $t^2$ and letting $t\to0$ yields $\int_Q\partial_i h\,\partial_i\zeta\d x\geq0$. Summing over $i$ proves \eqref{eq:projection-integral-sign}, since $\D w=p+\D\zeta$, $\D v-\D w=\D h$, and $\int_Q\D h\d x=0$ by periodicity.
\end{proof}

{\small
\bibliographystyle{abbrv-andre}
\bibliography{library}
}
\vspace{2em}
\noindent
{\small
\begin{minipage}[b]{\textwidth}
\textsc{Andrea Agazzi} \\
Department of Mathematics and Statistics, University of Bern \\
Alpeneggstrasse 22, 3012 Bern, CH \\
\texttt{andrea.agazzi@unibe.ch}
\end{minipage}

\vspace{1em}
\noindent
\begin{minipage}[b]{\textwidth}
\textsc{Kari Astala} \\
Department of Mathematics and Statistics, University of Helsinki \\
P.O. Box 68, 00014 University of Helsinki, Finland \\
\texttt{kari.astala@helsinki.fi}
\end{minipage}

\vspace{1em}
\noindent
\begin{minipage}[b]{\textwidth}
\textsc{Rodrigo Ba\~nuelos} \\
Department of Mathematics, Purdue University \\
West Lafayette, IN 47907, USA \\
\texttt{banuelos@purdue.edu}
\end{minipage}

\vspace{1em}
\noindent
\begin{minipage}[b]{\textwidth}
\textsc{Giuseppe Bruno} \\
Department of Mathematics and Statistics, University of Bern \\
Alpeneggstrasse 22, 3012 Bern, CH \\
\texttt{giuseppe.bruno@unibe.ch}
\end{minipage}

\vspace{1em}
\noindent
\begin{minipage}[b]{\textwidth}
\textsc{Gabriele Cassese} \\
Mathematical Institute, University of Oxford \\
Andrew Wiles Building, Woodstock Road, Oxford OX2 6GG, UK \\
\texttt{gabriele.cassese@maths.ox.ac.uk}
\end{minipage}

\vspace{1em}
\noindent
\begin{minipage}[b]{\textwidth}
\textsc{Daniel Faraco} \\
Departamento de Matem\'aticas, Universidad Aut\'onoma de Madrid \\
Ciudad Universitaria de Cantoblanco, 28049 Madrid, Spain \\
\textit{and} \\
Instituto de Ciencias Matem\'aticas, CSIC-UAM-UC3M-UCM, Madrid, Spain \\
\texttt{daniel.faraco@uam.es}
\end{minipage}

\vspace{1em}
\noindent
\begin{minipage}[b]{\textwidth}
\textsc{Andr\'e Guerra} \\
Department of Pure Mathematics and Mathematical Statistics, University of Cambridge \\
Wilberforce Rd, Cambridge CB3 0WB, UK \\
\texttt{adblg2@cam.ac.uk} \\
\textit{and} \\
Departamento de Matem\'atica, Instituto Superior T\'ecnico \\
Av. Rovisco Pais 1049-001 Lisboa, Portugal \\
\texttt{andre.l.guerra@tecnico.ulisboa.pt}
\end{minipage}

\vspace{1em}
\noindent
\begin{minipage}[b]{\textwidth}
\textsc{Mahdi Hormozi} \\
Beijing Institute of Mathematical Sciences and Applications, \\
Beijing 101408, China \\
\texttt{hormozi@bimsa.cn}
\end{minipage}

\vspace{1em}
\noindent
\begin{minipage}[b]{\textwidth}
\textsc{Bernd Kirchheim} \\
Mathematisches Institut, Universit\"at Leipzig \\
Augustusplatz 10, 04109 Leipzig, Germany \\
\texttt{kirchheim@math.uni-leipzig.de}
\end{minipage}

\vspace{1em}
\noindent
\begin{minipage}[b]{\textwidth}
\textsc{Aleksis Koski} \\
Department of Mathematics and Systems Analysis, Aalto University \\
P.O. Box 11100, FI-00076 Aalto, Finland \\
\texttt{aleksis.koski@aalto.fi}
\end{minipage}

\vspace{1em}
\noindent
\begin{minipage}[b]{\textwidth}
\textsc{Vjekoslav Kova\v{c}} \\
University of Zagreb Faculty of Science \\
Bijeni\v{c}ka cesta 30, 10000 Zagreb, Croatia \\
\texttt{vjekovac@math.hr}
\end{minipage}

\vspace{1em}
\noindent
\begin{minipage}[b]{\textwidth}
\textsc{Jan Kristensen} \\
Mathematical Institute, University of Oxford \\
Andrew Wiles Building, Woodstock Road, Oxford OX2 6GG, UK \\
\texttt{jan.kristensen@maths.ox.ac.uk}
\end{minipage}

\vspace{1em}
\noindent
\begin{minipage}[b]{\textwidth}
\textsc{Federico Pasqualotto} \\
Department of Mathematics, University of California San Diego \\
La Jolla, CA 92093, USA \\
\texttt{fpasqualotto@ucsd.edu}
\end{minipage}

\vspace{1em}
\noindent
\begin{minipage}[b]{\textwidth}
\textsc{Istv\'an Prause} \\
Mathematics, Faculty of Science and Engineering, \AA bo Akademi University \\
Tuomiokirkontori 3, 20500 Turku, Finland \\
\texttt{istvan.prause@abo.fi}
\end{minipage}

\vspace{1em}
\noindent
\begin{minipage}[b]{\textwidth}
\textsc{Eero Saksman} \\
Department of Mathematics and Statistics, University of Helsinki \\
P.O. Box 68, 00014 University of Helsinki, Finland \\
\texttt{eero.saksman@helsinki.fi}
\end{minipage}

\vspace{1em}
\noindent
\begin{minipage}[b]{\textwidth}
\textsc{Vladim\'ir \v{S}ver\'ak} \\
School of Mathematics, University of Minnesota \\
206 Church Street SE, Minneapolis, MN 55455, USA \\
\texttt{sverak@umn.edu}
\end{minipage}

\vspace{1em}
\noindent
\begin{minipage}[b]{\textwidth}
\textsc{L\'aszl\'o Sz\'ekelyhidi} \\
Max Planck Institute for Mathematics in the Sciences \\
Inselstrasse 22, 04103 Leipzig, Germany \\
\texttt{Laszlo.Szekelyhidi@mis.mpg.de}
\end{minipage}

\vspace{1em}
\noindent
\begin{minipage}[b]{\textwidth}
\textsc{Riccardo Tione} \\
Dipartimento di Matematica ``Giuseppe Peano'', Universit\`a di Torino \\
Via Carlo Alberto 10, 10123 Torino, Italy \\
\texttt{riccardo.tione@unito.it}
\end{minipage}

\vspace{1em}
\noindent
\begin{minipage}[b]{\textwidth}
\textsc{Alexander Volberg} \\
Department of Mathematics, Michigan State University \\
619 Red Cedar Road, East Lansing, MI 48824, USA \\
\texttt{volberg@math.msu.edu}
\end{minipage}

\vspace{1em}
\noindent
\begin{minipage}[b]{\textwidth}
\textsc{Hang Yang} \\
School of Mathematical Sciences \\
Fudan University \\
Shanghai 200433, China \\
\texttt{26210180112@m.fudan.edu.cn}
\end{minipage}\par
}

\end{document}